\pdfoutput=1
\documentclass[reqno]{amsart}

\usepackage{amsmath,amssymb,amsthm,amsfonts,enumerate, enumitem,hyperref,cleveref}
\usepackage{dsfont}
\usepackage[all]{xy}
\usepackage[T1]{fontenc}
\usepackage{tikz}
\usepackage{pgfplots}
\pgfplotsset{compat=1.15}
\usepackage{mathrsfs}
\usetikzlibrary{arrows}

\DeclareMathOperator{\Ker}{Ker}
\DeclareMathOperator{\im}{Im}
\DeclareMathOperator{\Aut}{Aut}

\DeclareMathOperator{\vfD}{Der^{\varphi}}
\DeclareMathOperator{\ivfD}{IDer^{\varphi}}

\DeclareMathOperator{\psiD}{Der^{\psi}}
\DeclareMathOperator{\ipsiD}{IDer^{\psi}}
\DeclareMathOperator{\apsiD}{ADer^{\psi}}
\DeclareMathOperator{\ppsiD}{PDer^{\psi}}

\theoremstyle{plain}
\newtheorem{theorem}{Theorem}[section]
\newtheorem{corollary}[theorem]{Corollary}
\newtheorem{proposition}[theorem]{Proposition}
\newtheorem{lemma}[theorem]{Lemma}

\theoremstyle{definition}
\newtheorem{definition}[theorem]{Definition}
\newtheorem{remark}[theorem]{Remark}
\newtheorem{example}[theorem]{Example}

\crefrangeformat{equation}{#3(#1)#4--#5(#2)#6}

\crefname{theorem}{Theorem}{Theorems}
\crefname{lemma}{Lemma}{Lemmas}
\crefname{corollary}{Corollary}{Corollaries}
\crefname{proposition}{Proposition}{Propositions}
\crefname{definition}{Definition}{Definitions}
\crefname{example}{Example}{Examples}
\crefname{remark}{Remark}{Remarks}
\crefname{conjecture}{Conjecture}{Conjectures}
\crefname{section}{Section}{Sections}
\crefname{equation}{\unskip}{\unskip}
\crefname{enumi}{\unskip}{\unskip}
\crefname{subsection}{Subsection}{Subsections}

\newcommand{\bt}{\beta}
\newcommand{\lb}{\lambda}

\newcommand{\gm}{\gamma}
\newcommand{\vf}{\varphi}

\newcommand{\dl}{\delta}
\newcommand{\Dl}{\Delta}

\newcommand{\sg}{\sigma}

\newcommand{\m}{{}^{-1}}
\newcommand{\sst}{\subseteq}
\newcommand{\impl}{\Rightarrow}

\newcommand{\wh}{\widehat}

\newcommand{\Dvf}{D^\varphi}
\newcommand{\Dpsi}{D^\psi}

\newcommand{\id}{\mathrm{id}}
\renewcommand{\iff}{\Leftrightarrow}

\begin{document}
	\title[Skew derivations of incidence algebras]{Skew derivations of incidence algebras}	
	\author{\'Erica Z. Fornaroli}
	\address{Departamento de Matem\'atica, Universidade Estadual de Maring\'a, Maring\'a, PR, CEP: 87020--900, Brazil}
	\email{ezancanella@uem.br}
	
	\author{Mykola Khrypchenko}
	\address{Departamento de Matem\'atica, Universidade Federal de Santa Catarina,  Campus Reitor Jo\~ao David Ferreira Lima, Florian\'opolis, SC, CEP: 88040--900, Brazil \and CMUP, Departamento de Matemática, Faculdade de Ciências, Universidade do Porto,
		Rua do Campo Alegre s/n, 4169--007 Porto, Portugal}
	\email{nskhripchenko@gmail.com}
	%\email[Corresponding author]{nskhripchenko@gmail.com}
	
	
	\begin{abstract}
		In the first part of the paper we describe $\vf$-derivations of the incidence algebra $I(X,K)$ of a locally finite poset $X$ over a field $K$, where $\vf$ is an arbitrary automorphism of $I(X,K)$. We show that they admit decompositions similar to that of usual derivations of $I(X,K)$.  In particular, the quotient of the space of $\vf$-derivations of $I(X,K)$ by the subspace of inner $\vf$-derivations of $I(X,K)$ is isomorphic to the first group of certain cohomology of $X$, which is developed in the second part of the paper.
	\end{abstract}
	
	\maketitle
	
	\tableofcontents
	
	\section*{Introduction}
	
	Let $K$ be a field and $A$ a $K$-algebra. Denote by $\Aut(A)$ the automorphism group of $A$ and fix $\vf\in\Aut(A)$. A linear map $d:A\to A$ is called a \textit{$\vf$-derivation} of $A$ if
	\begin{align*}
		d(ab)=d(a)b+\vf(a)d(b)
	\end{align*}
	for all $a,b\in A$. We denote by $\vfD(A)$ the $K$-space of all $\vf$-derivations of $A$. If $A$ is associative, then for any fixed $a\in A$ the map $\Dvf_a:A\to A$ defined by
	\begin{align*}
		\Dvf_a(b)=ab-\vf(b)a
	\end{align*}
	is a $\vf$-derivation called \textit{inner}. The subspace of all inner $\vf$-derivations of $A$ will be denoted by $\ivfD(A)$. A \textit{skew derivation} of $A$ is a map $A\to A$ which is a $\vf$-derivation of $A$ for some $\vf\in\Aut(A)$. Notice that the case $\vf=\id$ corresponds to usual derivations of $A$.

	The automorphisms~\cite{Stanley70} and derivations~\cite{Baclawski72,SpDo} of an incidence algebra $I(X,K)$ were fully characterized yet in the early 1970s. It was proved that any derivation of $I(X,K)$ is the sum of an inner derivation and a so-called \textit{additive} derivation~\cite{Baclawski72,SpDo}. This result was generalized to a wider class of algebras in~\cite{Kh-der}. Later, Jordan~\cite{Khr16,Xiao15} and Lie~\cite{Zhang-Khrypchenko} derivations of $I(X,K)$ were described modulo the (already known) description of usual derivations of $I(X,K)$. A connection between additive and inner derivations of $I(X,K)$ was recently studied in~\cite{FP1}.
	
	However, little has been known about the structure of skew derivations of $I(X,K)$. We can mention here some results on triangular algebras, although they are not always applicable to general incidence algebras. Benkovi\v{c} studied Jordan~\cite{Benkovic16} and Lie~\cite{Benkovic22} $\sg$-derivations of triangular algebras and obtained some descriptions involving usual $\sg$-derivations of such algebras. Regarding the $\sg$-derivations themselves, we could only find the preprint~\cite{Juana13}. The author describes (generalized) $\sg$-derivations of a triangular algebra $\mathrm{Tri}(A,M,B)$ in terms of the corresponding maps of the algebras $A$ and $B$ and certain linear maps on $M$. Unfortunately, this doesn't give a lot of information on skew derivations of $I(X,K)$, even when $I(X,K)$ is a triangular algebra. Inspired by~\cite{Bresar95}, $\sg$-biderivations and $\sg$-commuting maps of triangular algebras were investigated in~\cite{Candindo-Repka-Juana17}.
	
	In this paper we describe $\vf$-derivations of the incidence algebra $I(X,K)$ of a locally finite poset $X$ over a field $K$, where $\vf=\xi_\bt\circ M_\sg\circ \hat\lb$ is an arbitrary automorphism of $I(X,K)$ (see \cref{iso-decomp}). We soon reduce the problem to the case of automorphisms $\psi$ of the form $M_\sg\circ \hat\lb$ and prove in \cref{psiD=ipsiD+apsiD} that any $\psi$-derivation of $I(X,K)$ decomposes into the sum of an inner $\psi$-derivation and an \textit{additive} $\psi$-derivation, where the latter is defined in terms of the so-called \textit{$(\sg,\lb)$-additive elements} of $I(X,K)$. The decomposition is not unique in general, so a natural question is to describe those inner $\psi$-derivations which are additive. This is done in \cref{ppsiD=ipsiD-cap-apsiD}. As in the classical case, the quotient of $\vfD(I(X,K))$ by the subspace $\ivfD(I(X,K))$ admits a cohomological interpretation. The appropriate poset cohomology $H^n_{(\sg,\lb)}(X,K)$, which we call \textit{$(\sg,\lb)$-cohomology}, is developed in \cref{sec-cohom}. Then $\vfD(I(X,K))/\ivfD(I(X,K))$ is showed to be isomorphic to $H^1_{(\sg,\lb)}(X,K)$ (see \cref{vfD/ivfD-cong-H^1}). In the final part of the paper we study some properties of the $(\sg,\lb)$-cohomology of $X$ depending on $\sg$, $\lb$ and $X$ (see \cref{H^n_(sg_lb)-cong-H^n_(sg'_lb),all-comp=>H^1_(sg_lb)=0}). In particular, we provide in \cref{all-comp=>H^1_(sg_lb)=0} a necessary condition for $H^1_{(\sg,\lb)}(X,K)$ to be trivial, which corresponds to the situation when all the $\vf$-derivations of $I(X,K)$ are inner. The propositions are illustrated by several examples.
	
	\section{Preliminaries}\label{sec-prelim}
	
	\subsection{Skew derivations}
	We begin with a general auxiliary fact. Let $A$ be an associative algebra. For any invertible $u\in A$ we denote by $\xi_u$ the inner automorphism of $A$ given by 
	\begin{align}\label{xi_u(a)=uau^(-1)}
		\xi_u(a)=u a u\m    
	\end{align}
	for all $a\in A$.  
	
	\begin{proposition}
		Let $A$ be an associative algebra, $d\in\vfD(A)$ and $\psi\in\Aut(A)$ such that $\vf=\xi_u\circ\psi$. Then $d':A\to A$ given by
		\begin{align*}
			d'(a)=u\m d(a)
		\end{align*}
		is a $\psi$-derivation of $A$. Moreover, if $d=D^\vf_a$, then $d'=D^\psi_{u\m a}$.
	\end{proposition}
	\begin{proof}
		It is clear that $d'$ is linear, because $d$ is. Now, for all $a,b\in A$ we have
		\begin{align*}
			d'(ab)&=u\m d(ab)=u\m d(a)b+u\m\vf(a)d(b)\\
			&=u\m d(a)b+u\m u\psi(a)u\m d(b)=d'(a)b+\psi(a)d'(b).
		\end{align*}
		If $d=D^\vf_a$, then
		\begin{align*}
			d'(b)=u\m D^\vf_a(b)=u\m ab-u\m\vf(b)a=u\m ab-\psi(b)u\m a=D^\psi_{u\m a}(b).
		\end{align*}
	\end{proof}
	
	\begin{corollary}\label{vfD(A)=u.psiD(A)}
		Let $\psi\in\Aut(A)$ and $\vf=\xi_u\circ\psi$. Then 
		\begin{align*}
			\vfD(A)=u\cdot\psiD(A)\text{ and }\ivfD(A)=u\cdot\ipsiD(A).
		\end{align*}
	\end{corollary}
	
	\subsection{Posets and incidence algebras}
	
	Let $(X,\le)$ be a poset (i.e., partially ordered set). An \textit{interval} in $X$ is a subposet $\{z\in X : x\leq z\leq y\}$ for some $x\le y$ in $X$. A poset $X$ is said to be \textit{locally finite} whenever all its intervals are finite. An element $x_0\in X$ such that for any $x\in X$ either $x_0\le x$ or $x_0\ge x$ is called \textit{all-comparable}. A \textit{chain} in $X$ is a linearly ordered subposet of $X$. The \textit{length} of a chain $C\sst X$ is $l(C):=|C|-1$. The \textit{length} of $X$, denoted by $l(X)$, is the supremum of $l(C)$ over all finite chains $C\sst X$. An automorphism of $X$ is a bijective map $\lb:X\to X$ such that $x\le y\iff\lb(x)\le\lb(y)$ for all $x,y\in X$. The group of automorphisms of $X$ is denoted by $\Aut(X)$.
	
	For the rest of the paper we fix a locally finite poset $X$ and a field $K$. The \emph{incidence algebra} $I(X,K)$ of $X$ over $K$ (see~\cite{Rota64}) is the $K$-space of functions $f:X\times X\to K$ such that $f(x,y)=0$ for all $x\nleq y$ with multiplication
	$$
	(fg)(x,y)=\sum_{x\le z\le y}f(x,z)g(z,y)
	$$
	for all $f,g\in I(X,K)$. The algebra $I(X,K)$ is associative and unital. Its identity element is denoted by $\dl$. Recall that $\dl(x,y)=1$, if $x=y$, and $\dl(x,y)=0$, otherwise. Another remarkable element of $I(X,K)$ is the so-called \textit{zeta function} given by $\zeta(x,y)=1$ for all $x\le y$ (and $\zeta(x,y)=0$ for $x\nleq y$). For any $x\le y$ denote by $e_{xy}$ the element of $I(X,K)$ such that $e_{xy}(u,v)=1$, if $(u,v)=(x,y)$, and $e_{xy}(u,v)=0$, otherwise. Notice that $e_{xy}e_{zw}=e_{xw}$, if $y=z$, and $e_{xy}e_{zw}=0$, otherwise. We set $e_x:=e_{xx}$ for all $x\in X$. Then $\{e_x\}_{x\in X}$ are orthogonal idempotents of $I(X,K)$.
	
	By \cite[Theorem~1.2.3]{SpDo} an element $f\in I(X,K)$ is invertible if and only if $f(x,x)\neq 0$ for all $x\in X$. 
	
	\subsection{Automorphisms of $I(X,K)$}
	Since $I(X,K)$ is associative, any invertible $\bt\in I(X,K)$ defines the inner automorphism $\xi_\bt\in\Aut(I(X,K))$ as in \cref{xi_u(a)=uau^(-1)}.
	We recall another two classes of automorphisms of $I(X,K)$, whose definitions are given below.
	
	An element $\sg\in I(X,K)$ is \emph{multiplicative} if $\sg(x,y)\neq 0$, for all $x\le y$, and 
	\begin{align}\label{sg(x_z)=sg(x_y)sg(y_z)}
		\sg(x,z)=\sg(x,y)\sg(y,z)
	\end{align}
	whenever $x\le y\le z$. Every such $\sg$ determines the \emph{multiplicative automorphism} $M_{\sg}\in\Aut(I(X,K))$ by $M_{\sg}(f)=\sg\ast f$, for all $f\in I(X,K)$, where $\sg\ast f$ is the \textit{Hadamard product} given by 
	\begin{align}\label{(sg*f)(x_y)=sg(x_y)f(x_y)}
		(\sg\ast f)(x,y)=\sg(x,y)f(x,y)
	\end{align}
	for all $x,y\in X$. If $\sg, \tau \in I(X,K)$ are multiplicative, then so is $\sg\ast\tau$ and $M_\sg\circ M_\tau=M_{\sg*\tau}$ (see~\cite[7.3]{SpDo}). A multiplicative $\sg\in I(X,K)$ is called \textit{fractional} if there exists $\eta:X\to K^*$ such that $\sg(x,y)=\eta(x)\eta(y)\m$ for all $x\le y$. Recall from \cite[Proposition 7.3.3]{SpDo} that $M_\sg$ is inner if and only if $\sg$ is fractional.
	
	Any $\lb\in\Aut(X)$ \emph{induces} $\hat\lb\in\Aut(I(X,K))$ by means of
	\begin{align}\label{hat-lb(f)(x_y)=f(lb^(-1)(x)_lb^(-1)(x))}
		\hat\lb(f)(x,y)=f(\lb\m(x),\lb\m(y))
	\end{align}
	for all $f\in I(X,K)$ and $x\leq y$ in $X$.
	
	\begin{theorem}\cite[Theorem~7.3.6]{SpDo}\label{iso-decomp}
		Every $\vf\in\Aut(I(X,K))$ admits a decomposition $\vf=\xi_\bt\circ M_\sg\circ \hat\lb$ for some invertible $\bt\in I(X,K)$, multiplicative $\sg\in I(X,K)$ and $\lb\in\Aut(X)$.
	\end{theorem}

	\section{Skew derivations of $I(X,K)$}\label{sec-sg-der}

	\subsection{The description of $\vfD(I(X,K))$}
	
	Fix $\vf\in\Aut(I(X,K))$. By \cref{iso-decomp} we have $\vf=\xi_\bt\circ\psi$ with
	\begin{align}\label{psi=M_sg-circ-hat-lb}
		\psi=M_\sg\circ \hat\lb.
	\end{align}
	In view of \cref{vfD(A)=u.psiD(A)} it suffices to describe $\psiD(I(X,K))$, so we fix $\psi$ of the form \cref{psi=M_sg-circ-hat-lb} for the rest of the section.
	
	For any $f\in I(X,K)$ and $x\leq y$ in $X$, by \cref{(sg*f)(x_y)=sg(x_y)f(x_y),hat-lb(f)(x_y)=f(lb^(-1)(x)_lb^(-1)(x)),psi=M_sg-circ-hat-lb}, we have
	\begin{align}\label{vf(f)(x_y)}
		\psi(f)(x,y)=\sg(x,y)f(\lb^{-1}(x),\lb^{-1}(y))
	\end{align}
	and
	\begin{align}\label{vf(e_xy)}
		\psi(e_{xy})=\sg(\lb(x),\lb(y))e_{\lb(x)\lb(y)}.
	\end{align}
	
	We introduce a class of maps which will take part in the description of $\psi$-derivations of $I(X,K)$.
	\begin{definition}\label{additive_element}
		Let $\sg\in I(X,K)$ be multiplicative and $\lb\in\Aut(X)$. An element $\tau\in I(X,K)$ is called \emph{$(\sg,\lb)$-additive} if 
		\begin{align}\label{tau(x_y)=0-for-lb(x)-ne-y}
			\tau(x,y)=0,\text{ whenever }\lb(x)\not\le y,
		\end{align} 
		and 
		\begin{align}\label{tau_xyz}
			\tau(x,z)=\tau(x,y)+\sg(\lb(x),\lb(y))\tau(y,z)
		\end{align}
		for all $x\le y\le z$.
	\end{definition}
	
	Observe that $(\sg,\lb)$-additive elements form a $K$-subspace of $I(X,K)$.
	
	\begin{definition}\label{(sg_lb)-additive}
		Given a $(\sg,\lb)$-additive $\tau\in I(X,K)$, define $L_{\tau}:I(X,K)\to I(X,K)$ by 
		\begin{align}\label{(b-inv-d(f))(lb(x)_y)=tau(x_y)f(x_y)}
			L_{\tau}(f)(\lb(x),y)=\tau(x,y)f(x,y),
		\end{align} 
		for all $f\in I(X,K)$ and $x,y\in X$.    
	\end{definition}
	Note that if $\lb(x)\nleq y$, then $L_{\tau}(f)(\lb(x),y)=0$ because $\tau(x,y)=0$. Thus $L_{\tau}(f)\in I(X,K)$. We also have $L_{\tau}(f)(\lb(x),y)=0$ whenever $x\nleq y$, since $f\in I(X,K)$. 
	
	\begin{remark}\label{L_tau(e_xy)}
		Observe that for all $x\le y$
		\begin{align}\label{L_tau(e_xy)-defn}
			L_{\tau}(e_{xy})=
			\begin{cases}
				\tau(x,y)e_{\lb(x)y}, & \text{if }\lb(x)\le y,\\
				0, &\text{otherwise}.
			\end{cases}
		\end{align}
		In particular, $L_{\tau}=L_{\tau'}$ if and only if $\tau=\tau'$.
	\end{remark}
	
	\begin{proposition}\label{L_tau-in-vfD}
		If $\tau\in I(X,K)$ is $(\sg,\lb)$-additive, then $L_{\tau}\in \psiD(I(X,K))$.
	\end{proposition}
	\begin{proof}
		Clearly, $ L_{\tau}$ is linear. Let $f,g\in I(X,K)$ and $x,y\in X$ with $\lb(x)\leq y$. Then 
		\begin{align}
			(L_{\tau}(f)g+\psi(f)L_{\tau}(g))(\lb(x),y) & = \sum_{\lb(x)\leq z\leq y} L_{\tau}(f)(\lb(x),z)g(z,y) \nonumber\\
			& \quad +\sum_{\lb(x)\leq \lb(z)\leq y}\psi(f)(\lb(x),\lb(z))L_{\tau}(g)(\lb(z),y). \label{Ltau1}
		\end{align}
		
		Suppose $x\nleq y$. In this case, $L_{\tau}(f)(\lb(x),z)=0$ for all $z\in X$ such that $\lb(x)\leq z\leq y$, because $x\nleq z$. Therefore the first summation in \cref{Ltau1} is zero. Moreover, if $\lb(x)\leq \lb(z)\leq y$, then $z\nleq y$, because otherwise $\lb(x)\leq \lb(z)\leq\lb(y)$ and so $x\leq y$, a contradiction. Thus $L_{\tau}(g)(\lb(z),y)=0$ and, therefore, the second summation in \cref{Ltau1} is zero. It follows that $(L_{\tau}(f)g+\psi(f)L_{\tau}(g))(\lb(x),y)=0=L_{\tau}(fg)(\lb(x),y)$.
		
		Now, consider $x\leq y$. Then, using \cref{vf(f)(x_y),(b-inv-d(f))(lb(x)_y)=tau(x_y)f(x_y)}, we rewrite \cref{Ltau1} as
		\begin{align}
			(L_{\tau}(f)g+\psi(f)L_{\tau}(g))(\lb(x),y) & = \sum_{{\lb(x)\leq z\leq y}\atop{x\leq z}} \tau(x,z)f(x,z)g(z,y) \nonumber\\
			& \quad +\sum_{{\lb(x)\leq \lb(z)\leq y}\atop{z\leq y}}\sg(\lb(x),\lb(z))f(x,z)\tau(z,y)g(z,y)\nonumber\\
			& = \sum_{{\lb(x)\leq z\leq y}\atop{x\leq z\leq y}} \tau(x,z)f(x,z)g(z,y) \nonumber\\
			& \quad +\sum_{{\lb(x)\leq \lb(z)\leq y}\atop{x\leq z\leq y}}\sg(\lb(x),\lb(z))\tau(z,y)f(x,z)g(z,y). \label{Ltau2}
		\end{align}
		If $\lb(x)\nleq z$, then $\tau(x,z)=0$, and if $\lb(z)\nleq y$, then $\tau(z,y)=0$. Thus, by \cref{Ltau2,tau_xyz}, 
		\begin{align*}
			(L_{\tau}(f)g+\psi(f)L_{\tau}(g))(\lb(x),y) & = \sum_{x\leq z\leq y} \tau(x,z)f(x,z)g(z,y)\\
			& \quad +\sum_{x\leq z\leq y}\sg(\lb(x),\lb(z))\tau(z,y)f(x,z)g(z,y)\\
			& = \sum_{x\leq z\leq y} [\tau(x,z)+\sg(\lb(x),\lb(z))\tau(z,y)]f(x,z)g(z,y)\\
			& = \sum_{x\leq z\leq y} \tau(x,y)f(x,z)g(z,y)\\
			& = \tau(x,y)(fg)(x,y)= L_{\tau}(fg)(\lb(x),y).
		\end{align*}
		Since $\lb\in \Aut(X)$, it follows that $L_{\tau}(fg)= L_{\tau}(f)g+\psi(f)L_{\tau}(g)$.
	\end{proof}
	
	\begin{definition}\label{vf_additive}
		A $\psi$-derivation $d$ of $I(X,K)$ will be called \emph{additive}, if there exists a $(\sg,\lb)$-additive $\tau\in I(X,K)$ such that $d=L_{\tau}$.
	\end{definition}
	
	Observe that $\tau\in I(X,K)$ from \cref{vf_additive} is unique by \cref{L_tau(e_xy)}.
	
	We denote the set of all additive $\psi$-derivations of $I(X,K)$ by $\apsiD(I(X,K))$. Clearly, $\apsiD(I(X,K))$ is a $K$-subspace of $\psiD(I(X,K))$.
	
	\begin{lemma}\label{d(e_x)=0iff}
		A $\psi$-derivation $d$ of $I(X,K)$ is additive if and only if $d(e_x)=0$ for all $x\in X$.
	\end{lemma}
	\begin{proof}
		\textit{The ``if'' part.} Assume that $d(e_x)=0$ for all $x\in X$. Define $\tau\in I(X,K)$ by
		\begin{align}\label{definition_tau}
			\tau(x,y)=d(e_{xy})(\lb(x),y), \text{ for all } x\leq y.
		\end{align}
		Clearly, $\tau$ satisfies \cref{tau(x_y)=0-for-lb(x)-ne-y}. Let $x\leq y$ in $X$. Then, by \cref{vf(e_xy)} and the assumption on $d$,
		\begin{align*}
			d(e_{xy}) & = d(e_xe_{xy}e_y)=d(e_x)e_{xy}+\psi(e_x)d(e_{xy}e_y)\\
			& = \psi(e_x)d(e_{xy})e_y+\psi(e_x)\psi(e_{xy})d(e_y) = e_{\lb(x)}d(e_{xy})e_y,
		\end{align*}
		and therefore
		\begin{align}\label{d(e_xy)=tau(x_y)e_lb(x)y}
			d(e_{xy})& =\begin{cases}
				\tau(x,y)e_{\lb(x)y}, & \text{if }\lb(x)\leq y,\\
				0, & \text{otherwise}.
			\end{cases}
		\end{align}
		
		To show that $\tau$ satisfies \cref{tau_xyz}, let $x\le y\le z$ in $X$. Assume first that $\lb(x)\nleq z$. Then clearly $\lb(x)\nleq y$. Moreover, if we have $\lb(y)\leq z$, then $\lb(x)\leq \lb(y)\leq z$, a contradiction with $\lb(x)\nleq z$. So $\lb(y)\nleq z$. Thus, by \cref{tau(x_y)=0-for-lb(x)-ne-y},
		$$\tau(x,z)=0=\tau(x,y)+\sg(\lb(x),\lb(y))\tau(y,z).$$
		
		Now, suppose $\lb(x)\leq z$. Then 
		\begin{align}\label{tau8}
			\tau(x,z)e_{\lb(x)z} = d(e_{xz})=d(e_{xy}e_{yz}) = d(e_{xy})e_{yz}+ \psi(e_{xy})d(e_{yz}). 
		\end{align}
		We have $3$ cases to consider. 
		
		\textit{Case 1.} $\lb(x)\leq y$.
		By \cref{d(e_xy)=tau(x_y)e_lb(x)y,tau8},
		\begin{align}\label{tau9}
			\tau(x,z)e_{\lb(x)z} & = \tau(x,y)e_{\lb(x)y}e_{yz}+\psi(e_{xy})d(e_{yz}).
		\end{align}
		If $\lb(y)\leq z$, then by \cref{vf(e_xy),tau9,d(e_xy)=tau(x_y)e_lb(x)y} 
		\begin{align*}
			\tau(x,z)e_{\lb(x)z}&=\tau(x,y)e_{\lb(x)z}+\sg(\lb(x),\lb(y))e_{\lb(x)\lb(y)}\tau(y,z)e_{\lb(y)z}\\ &=\tau(x,y)e_{\lb(x)z}+\sg(\lb(x),\lb(y))\tau(y,z)e_{\lb(x)z}.
		\end{align*}
		Otherwise, by \cref{tau9,d(e_xy)=tau(x_y)e_lb(x)y,tau(x_y)=0-for-lb(x)-ne-y}
		\begin{align*}
			\tau(x,z)e_{\lb(x)z}&=\tau(x,y)e_{\lb(x)z}+0
			= \tau(x,y)e_{\lb(x)z}+\sg(\lb(x),\lb(y))\tau(y,z)e_{\lb(x)z}.
		\end{align*}
		Therefore, \cref{tau_xyz} holds. %$\tau(x,z)=\tau(x,y)+\sg(\lb(x),\lb(y))\tau(y,z)$.
		
		\textit{Case 2.} $\lb(x)\nleq y$ and $\lb(y)\leq z$.
		In this case, by \cref{tau8} we have,
		\begin{align*}
			\tau(x,z)e_{\lb(x)z} & = 0+\sg(\lb(x),\lb(y))\tau(y,z)e_{\lb(x)z}\\
			& = [\tau(x,y)+\sg(\lb(x),\lb(y))\tau(y,z)]e_{\lb(x)z},
		\end{align*}
		whence \cref{tau_xyz}. 
		
		\textit{Case 3.} $\lb(x)\nleq y$ and $\lb(y)\nleq z$.
		In this case, $\tau(x,y)=\tau(y,z)=0$ by \cref{tau(x_y)=0-for-lb(x)-ne-y} and hence by \cref{d(e_xy)=tau(x_y)e_lb(x)y,tau8}, 
		\begin{align*}
			\tau(x,z)e_{\lb(x)z} = 0+0 =[\tau(x,y)+\sg(\lb(x),\lb(y))\tau(y,z)]e_{\lb(x)z}.
		\end{align*}
		Thus, \cref{tau_xyz} holds.
		
		It follows that $\tau$ is $(\sg,\lb)$-additive. Let $f\in I(X,K)$ and $x,y\in X$ with $\lb(x)\leq y$. If $x\leq y$, then using $d(e_x)=d(e_y)=0$ we have
		\begin{align}
			d(f)(\lb(x),y) e_{\lb(x)y} & = e_{\lb(x)}d(f)e_y = \psi(e_x)d(f)e_{y}\nonumber\\
			& =\psi(e_x)d(fe_{y})=d(e_xfe_{y})\label{exfey=0}\\
			& = d(f(x,y)e_{xy})=f(x,y)d(e_{xy})\nonumber\\
			& = f(x,y)\tau(x,y)e_{\lb(x)y}\nonumber,
		\end{align}
		where the last equality is due to \cref{d(e_xy)=tau(x_y)e_lb(x)y}. Thus, 
		$$
		d(f)(\lb(x),y)= f(x,y)\tau(x,y)=L_{\tau}(f)(\lb(x),y).
		$$
		And if $x\nleq y$, then by \cref{exfey=0}, $d(f)(\lb(x),y)=0=L_{\tau}(f)(\lb(x),y)$. Since $\lb\in \Aut(X)$, it follows that $d(f)=L_{\tau}(f)$. Thus  $d=L_{\tau}$ is additive.
		
		\textit{The ``only if'' part.} Suppose there is a $(\sg,\lb)$-additive $\tau \in I(X,K)$ such that $d=L_{\tau}$. By \cref{tau_xyz}, $\tau(x,x)=0$ for all $x\in X$. Thus, given $u,v\in X$ with $\lb(u)\leq v$, we have
		$$ d(e_x)(\lb(u),v) =\tau(u,v)e_x(u,v)=0. $$
		Therefore, $d(e_x)=0$ for all $x\in X$.
	\end{proof}
	
	\begin{lemma}\label{d(e_x)=ad_f(e_x)}
		Let $d\in \psiD(I(X,K))$ and define $f\in I(X,K)$ by $f(x,y)=d(e_y)(x,y)$ for all $x\le y$. Then $d(e_x)=\Dpsi_f(e_x)$ for all $x\in X$.
	\end{lemma}
	\begin{proof}
		For each $x\in X$,
		\begin{align}\label{adfe_x}
			\Dpsi_f(e_x)=fe_x-\psi(e_x)f
		\end{align}
		and
		\begin{align}\label{de_x}
			d(e_x)=d(e_x)e_x+\psi(e_x)d(e_x).
		\end{align}
		
		Let $u\leq v$ in $X$. Then 
		\begin{align*}
			(fe_x)(u,v) &=\begin{cases}
				0, & \text{if } v\neq x\\
				f(u,x), & \text{if } v=x
			\end{cases}\\
			&=\begin{cases}
				0, & \text{if } v\neq x\\
				d(e_x)(u,x), & \text{if } v=x
			\end{cases}\\
			&=(d(e_x)e_x)(u,v).
		\end{align*}
		Therefore
		\begin{align}\label{fe_x=d(e_x)e_x}
			fe_x=d(e_x)e_x. 
		\end{align}
		By \cref{de_x}, for all $x\in X$,
		\begin{align}\label{vf(e_x)de_x}
			\psi(e_x)d(e_x)e_x=0,
		\end{align}
		thus
		\begin{align}\label{vf(e_x)fe_x}
			\psi(e_x)fe_x=0,
		\end{align}
		by \cref{fe_x=d(e_x)e_x,vf(e_x)de_x}. 
		
		If $u\leq x$, then 
		\begin{align*}
			(\psi(e_x)f+\psi(e_x)d(e_x))(u,x)e_{ux} = e_u\psi(e_x)fe_x+e_u\psi(e_x)d(e_x)e_x=0,
		\end{align*}
		by \cref{vf(e_x)de_x,vf(e_x)fe_x}.
		And if $u\leq v\neq x$, then by \cref{de_x,fe_x=d(e_x)e_x}, 
		\begin{align*}
			(\psi(e_x)f+\psi(e_x)d(e_x))(u,v)e_{uv} & = e_u\psi(e_x)fe_v+e_u\psi(e_x)d(e_x)e_v\\
			& = e_u\psi(e_x)d(e_v)e_v+e_u(d(e_x)-d(e_x)e_x)e_v\\
			& = e_u\psi(e_x)(d(e_v)-\psi(e_v)d(e_v))+e_ud(e_x)e_v\\
			& = e_u\psi(e_x)d(e_v)+e_ud(e_x)e_v\\
			& = e_u(d(e_x)e_v+\psi(e_x)d(e_v))\\
			& = e_ud(e_xe_v)=0.
		\end{align*}
		Thus, $\psi(e_x)f+\psi(e_x)d(e_x)=0$, that is,
		\begin{align}\label{vf(e_x)d(e_x)=-vf(e_x)f}
			\psi(e_x)d(e_x)=-\psi(e_x)f.
		\end{align}
		Therefore, $d(e_x)=\Dpsi_f(e_x)$, by \cref{fe_x=d(e_x)e_x,vf(e_x)d(e_x)=-vf(e_x)f,adfe_x,de_x}.
	\end{proof}
	
	\begin{theorem}\label{psiD=ipsiD+apsiD}
		Let $\psi=M_\sg\circ \hat\lb$. Then $\psiD(I(X,K))=\ipsiD(I(X,K))+\apsiD(I(X,K))$.
	\end{theorem}
	\begin{proof}
		Let $d\in \psiD(I(X,K))$. By \cref{d(e_x)=ad_f(e_x)}, there exists $f\in I(X,K)$ such that $d(e_x)=\Dpsi_f(e_x)$ for all $x\in X$. Thus, $d_1:=d-\Dpsi_f$ is an additive $\psi$-derivation by \cref{d(e_x)=0iff}, and therefore $d=\Dpsi_f+d_1\in \ipsiD(I(X,K))+\apsiD(I(X,K))$.
	\end{proof}
	
	\cref{psiD=ipsiD+apsiD,vfD(A)=u.psiD(A)} imply the following.
	\begin{corollary}\label{vfD=ivfD+bt.avfD}
		Let $\vf=\xi_\bt\circ \psi$ with $\psi=M_\sg\circ \hat\lb$. Then $\vfD(I(X,K))=\ivfD(I(X,K))+\bt\cdot\apsiD(I(X,K))$.
	\end{corollary}
	
	\subsection{The description of $\ivfD(I(X,K))\cap\bt\cdot\apsiD(I(X,K))$}
	
	Observe that the sum $\ivfD(I(X,K))+\bt\cdot\apsiD(I(X,K))$ from \cref{vfD=ivfD+bt.avfD} is not direct in general. Our next goal is to describe the intersection $\ivfD(I(X,K))\cap\bt\cdot\apsiD(I(X,K))$. To this end, we first describe the elements of $\ipsiD(I(X,K))\cap\apsiD(I(X,K))$.
	
	\begin{lemma}\label{d_f_aditiva_tau}
		If $\Dpsi_f=L_{\tau}$, then, for any $x\leq y$ in X, 
		$$\tau(x,y)= f(\lb(x),x)-\sg(\lb(x),\lb(y))f(\lb(y),y).$$
	\end{lemma}
	\begin{proof}
		%If $x=y$, then by \cref{tau_xyz}, $\tau(x,x)=0=f(\lb(x),x)-\sg(\lb(x),\lb(x))f(\lb(x),x)$.
		Suppose $x\le y$. If $\lb(x)\nleq y$, then $\tau(x,y)=0$, by \cref{tau(x_y)=0-for-lb(x)-ne-y}. Moreover, in this case, $\lb(x)\nleq x$ and $\lb(y)\nleq y$, therefore $f(\lb(x),x)-\sg(\lb(x),\lb(y))f(\lb(y),y)=0$. 
		
		If $\lb(x)\leq y$, then, by \cref{vf(e_xy),(b-inv-d(f))(lb(x)_y)=tau(x_y)f(x_y)},
		\begin{align*}
			\tau(x,y) & = \tau(x,y)e_{xy}(x,y)=L_{\tau}(e_{xy})(\lb(x),y)  \\
			& =\Dpsi_f(e_{xy})(\lb(x),y)=(f e_{xy}-\psi(e_{xy})f)(\lb(x),y) \\
			& = (f e_{xy})(\lb(x),y)-\sg(\lb(x),\lb(y))(e_{\lb(x)\lb(y)} f)(\lb(x),y)\\
			& = f(\lb(x),x)-\sg(\lb(x),\lb(y))f(\lb(y),y). 
		\end{align*}
	\end{proof}
	
	\begin{definition}\label{funcao_potencial}
		Let $\sg\in I(X,K)$ be multiplicative, $\lb\in\Aut(X)$ and $\epsilon:X\to K$ such that
		\begin{align}\label{eps(x)=0-for-lb(x)-nleq-x}
			\epsilon(x)=0,\text{ if }\lb(x)\nleq x.    
		\end{align} 
		We define $\tau_{\epsilon}\in I(X,K)$ by
		\begin{align}\label{tau_epsilon-defn}
			\tau_{\epsilon}(x,y) = \epsilon(x)-\sg(\lb(x),\lb(y))\epsilon(y)
		\end{align}
		for all $x\leq y$. Given $\tau\in I(X,K)$, we say that $\tau$ is \emph{$(\sg,\lb)$-potential}, if $\tau=\tau_\epsilon$ for some $\epsilon:X\to K$ satisfying \cref{eps(x)=0-for-lb(x)-nleq-x}.
	\end{definition}
	
	\begin{proposition}
		Every $(\sg,\lb)$-potential $\tau\in I(X,K)$ is $(\sg,\lb)$-additive.
	\end{proposition}
	\begin{proof}
		Take $\epsilon:X\to K$ satisfying \cref{eps(x)=0-for-lb(x)-nleq-x}. Let us show that $\tau_{\epsilon}$ is $(\sg,\lb)$-additive. 
		
		Let $x,y\in X$ such that $\lb(x)\nleq y$. If $x\nleq y$, then $\tau_{\epsilon}(x,y)=0$ because $\tau_{\epsilon}\in I(X,K)$, and if $x\leq y$, then $\lb(x)\nleq x$ and $\lb(y)\nleq y$. Hence $\epsilon(x)=\epsilon(y)=0$ and $\tau_{\epsilon}(x,y)=0$. Thus \cref{tau(x_y)=0-for-lb(x)-ne-y} is satisfied for $\tau=\tau_{\epsilon}$. Let $x\le y\le z$. Let us show that \cref{tau_xyz} also holds. Using \cref{sg(x_z)=sg(x_y)sg(y_z)}, we have
		\begin{align*}
			\tau_{\epsilon}(x,y)+\sg(\lb(x),\lb(y))\tau_{\epsilon}(y,z) & = \epsilon(x)-\sg(\lb(x),\lb(y))\epsilon(y)\\
			& \quad +\sg(\lb(x),\lb(y))[\epsilon(y)-\sg(\lb(y),\lb(z))\epsilon(z)] \\
			& = \epsilon(x)-\sg(\lb(x),\lb(z))\epsilon(z)= \tau_{\epsilon}(x,z),
		\end{align*}
		as desired.
	\end{proof}

	Observe that $(\sg,\lb)$-potential elements form a $K$-subspace of the space of $(\sg,\lb)$-additive elements.

	\begin{definition}
		An additive $\psi$-derivation $L_{\tau}$ such that $\tau$ is $(\sg,\lb)$-potential will be called \emph{potential}.
	\end{definition}
	
	Denote by $\ppsiD(I(X,K))$ the subspace of all potential $\psi$-derivations of $I(X,K)$.
	
	\begin{theorem}\label{ppsiD=ipsiD-cap-apsiD}
		Let $\psi=M_\sg\circ \hat\lb$. Then $\ppsiD(I(X,K))=\ipsiD(I(X,K))\cap\apsiD(I(X,K))$.
	\end{theorem}
	\begin{proof}
		Let $d\in \ipsiD(I(X,K))\cap\apsiD(I(X,K))$. Then there exists $f\in I(X,K)$ and a $(\sg,\lb)$-additive $\tau\in I(X,K)$ such that $d=\Dpsi_f=L_{\tau}$, and $\tau$ is as in \cref{d_f_aditiva_tau}. Then define $\epsilon: X\to K$ by $\epsilon(x)= f(\lb(x),x)$. If $\lb(x)\nleq x$, then $\epsilon(x)=0$. Moreover, 
		$$
		\tau(x,y) = \epsilon(x)-\sg(\lb(x),\lb(y))\epsilon(y)=\tau_{\epsilon}(x,y),
		$$ 
		for all $x\leq y$ in $X$. Thus, $\tau=\tau_{\epsilon}$ is $(\sg,\lb)$-potential and therefore $d\in \ppsiD(I(X,K))$.
		
		Conversely, let $d\in \ppsiD(I(X,K))$. Then there exists a $(\sg,\lb)$-potential $\tau_{\epsilon}$ such that $d=L_{\tau_{\epsilon}}$. We only need to prove that $d\in \ipsiD(I(X,K))$. 
		
		Let $h:X\times X\to K$ be defined by 
		$$h(x,y)=\begin{cases}
			\epsilon(y), & \text{if } x=\lb(y),\\
			0, & \text{otherwise}.
		\end{cases}$$
		Note that if $\lb(y)\nleq y$, then $h(\lb(y),y)= \epsilon(y)=0$ by \cref{eps(x)=0-for-lb(x)-nleq-x}, thus $h\in I(X,K)$. 
		
		Let $f\in I(X,K)$ and $x,y\in X$ with $\lb(x)\leq y$. Then
		\begin{align}
			D_h^\psi(f)(\lb(x),y) & = (hf)(\lb(x),y)-(\psi(f)h)(\lb(x),y) \nonumber \\
			& = \sum_{\lb(x)\leq z\leq y} h(\lb(x),z)f(z,y)-\sum_{\lb(x)\leq z\leq y}\psi(f)(\lb(x),z)h(z,y). \label{naosei}
		\end{align}
		If $\lb(x)\nleq x$, then $h(\lb(x),z)=0$ for all $z\in X$ such that $\lb(x)\leq z\leq y$, therefore the first sum in \cref{naosei} is zero, and if $\lb(x)\leq x$, then it is equal to $h(\lb(x),x)f(x,y)=\epsilon(x)f(x,y)$. Analogously, if $\lb(y)\nleq y$, then $h(z,y)=0$ for all $z\in X$ such that $\lb(x)\leq z\leq y$, therefore the second sum in \cref{naosei} is zero, and it is equal to $\psi(f)(\lb(x),\lb(y))h(\lb(y),y)=\psi(f)(\lb(x),\lb(y))\epsilon(y)$ otherwise. Thus, since $\epsilon(x)=0$ if $\lb(x)\nleq x$ and $\epsilon(y)=0$ if $\lb(y)\nleq y$ by \cref{eps(x)=0-for-lb(x)-nleq-x}, then by \cref{vf(f)(x_y),tau_epsilon-defn,(b-inv-d(f))(lb(x)_y)=tau(x_y)f(x_y)},
		\begin{align*}
			D_h^\psi(f)(\lb(x),y) & = \epsilon(x)f(x,y)-\psi(f)(\lb(x),\lb(y))\epsilon(y) \\
			& = \epsilon(x)f(x,y)-\sigma(\lb(x),\lb(y))f(x,y)\epsilon(y) \\
			& = \tau_{\epsilon}(x,y)f(x,y) = L_{\tau_{\epsilon}}(f)(\lb(x),y).                    
		\end{align*}
		Since $\lb\in \Aut(X)$, it follows that $ L_{\tau_{\epsilon}}(f)= D_h^\psi(f)$ and, therefore, $d=L_{\tau_{\epsilon}}=D_h^\psi \in \ipsiD(I(X,K))$. 
	\end{proof}
	
	\begin{corollary}\label{pvfD=ivfD-cap-avfD}
		Let $\vf=\xi_\bt\circ \psi$ with $\psi=M_\sg\circ \hat\lb$. Then $\bt\cdot\ppsiD(I(X,K))=\ivfD(I(X,K))\cap\bt\cdot\apsiD(I(X,K))$.
	\end{corollary}

	\section{Poset skew cohomology}\label{sec-cohom}
	
	\subsection{The construction}
	\cref{pvfD=ivfD-cap-avfD,vfD=ivfD+bt.avfD} show that the quotient of the space $\vfD(I(X,K))$ by the subspace $\ivfD(I(X,K))$ admits a cohomological interpretation. In this section we introduce the corresponding cohomology.
	
	Fix $\lb\in\Aut(X)$. For an arbitrary integer $n>0$ denote
	\begin{align*}
		X^n_\le&=\{(x_0,\dots,x_{n-1})\in X^n: x_0\le x_1\le\dots \le x_{n-1}\}.
		%X^n_{\le,\lb}&=\{(x_1,\dots,x_n)\in X_\le^n: \lb(x_1)\le x_n\}.
	\end{align*}
	Further, for any integer $n\ge 0$ introduce the following $K$-spaces:
	\begin{align*}
		C^n(X,K)&=\{f:X^{n+1}_\le\to K\},\\
		C^n_\lb(X,K)&=\{f\in C^n(X,K): f(x_0,\dots,x_n)=0\text{ if }\lb(x_0)\nleq x_n\}.
	\end{align*}
	The elements of $C^n(X,K)$ (resp.~$C^n_\lb(X,K)$) will be called \textit{cochains} (resp.~{\it $\lb$-cochains}) {\it of degree $n$} of $X$ with values in $K$. We will usually omit ``with values in $K$'', when $K$ is clear from the context. 
	
	Fix, moreover, a multiplicative $\sg\in I(X,K)$. Define $\dl^n_{(\sg,\lb)}:C^n(X,K)\to C^{n+1}(X,K)$, $f\mapsto\dl^n_{(\sg,\lb)}f$, as follows:
	\begin{align}
		(\dl^n_{(\sg,\lb)}f)(x_0,\dots,x_{n+1})&=\sg(\lb(x_0),\lb(x_1))f(x_1,\dots,x_{n+1})\notag\\
		&\quad+\sum_{i=1}^{n+1}(-1)^i f(x_0,\dots,\wh{x_i},\dots,x_{n+1}),\label{dl^nf(x_1...x_(n+1))}
	\end{align}
	where $x_0,\dots,\wh{x_i},\dots,x_{n+1}$ means $x_0,\dots,x_{n+1}$ with $x_i$ removed.
	
	\begin{lemma}\label{dl^n_lb(C^n_lb)-sst-C^(n+1)_lb}
		For all $n\ge 0$ we have $\dl^n_{(\sg,\lb)}(C^n_\lb(X,K))\sst C^{n+1}_\lb(X,K)$.
	\end{lemma}
	\begin{proof}
		Let $(x_0,\dots,x_{n+1})\in X^{n+2}_\le$ with $\lb(x_0)\nleq x_{n+1}$. Then $\lb(x_1)\nleq x_{n+1}$, since otherwise $\lb(x_0)\le \lb(x_1)\le x_{n+1}$, and $\lb(x_0)\nleq x_n$, since otherwise $\lb(x_0)\leq x_n\le x_{n+1}$. Hence, for any $f\in C^n_\lb(X,K)$ and $1\le i\le n+1$, one has
		\begin{align*}
			f(x_1,\dots,x_{n+1})=f(x_0,\dots,\wh{x_i},\dots,x_{n+1})=0,
		\end{align*}
		so $(\dl^n_{(\sg,\lb)}f)(x_0,\dots,x_{n+1})=0$ by \cref{dl^nf(x_1...x_(n+1))}.
	\end{proof}
	
	To prove that $\dl^{n+1}_{(\sg,\lb)}\circ \dl^n_{(\sg,\lb)}=0$, we use the ideas from \cite{Loday92}. For any $0\le i\le n+1$ define $\dl^n_i:C^n(X,K)\to C^{n+1}(X,K)$, $f\mapsto \dl^n_if$, by
	\begin{align*}
		(\dl^n_0f)(x_0,\dots,x_{n+1})&=\sg(\lb(x_0),\lb(x_1))f(x_1,\dots,x_{n+1}),\\
		(\dl^n_if)(x_0,\dots,x_{n+1})&=f(x_0,\dots,\wh{x_i},\dots,x_{n+1}),\ 1\le i\le n+1.
	\end{align*}
	
	\begin{lemma}\label{dl_i.dl_j=dl_(j+1).dl_i}
		For all $n\ge 0$ and $0\le i\le j\le n+1$ we have $\dl^{n+1}_i\circ \dl^n_j=\dl^{n+1}_{j+1}\circ \dl^n_i$.	
	\end{lemma}
	\begin{proof}
		Take arbitrary $f\in C^n(X,K)$ and $(x_0,\dots,x_{n+2})\in X^{n+3}_\le$.
		
		\textit{Case 1.} $0=i=j$. Then
		\begin{align*}
			(\dl^{n+1}_i\dl^n_jf)(x_0,\dots,x_{n+2})&=\sg(\lb(x_0),\lb(x_1))(\dl^n_jf)(x_1,\dots,x_{n+2})\\
			&=\sg(\lb(x_0),\lb(x_1))\sg(\lb(x_1),\lb(x_2))f(x_2,\dots,x_{n+2})\\
			&=\sg(\lb(x_0),\lb(x_2))f(x_2,\dots,x_{n+2})\\
			&=(\dl^n_if)(x_0,x_2,\dots,x_{n+2})\\
			&=(\dl^{n+1}_{j+1}\dl^n_if)(x_0,\dots,x_{n+2}).
		\end{align*}
		
		\textit{Case 2.} $0=i<j$. Then
		\begin{align*}
			(\dl^{n+1}_i\dl^n_jf)(x_0,\dots,x_{n+2})&=\sg(\lb(x_0),\lb(x_1))(\dl^n_jf)(x_1,\dots,x_{n+2})\\
			&=\sg(\lb(x_0),\lb(x_1))f(x_1,\dots,\wh{x_{j+1}},\dots,x_{n+2})\\
			&=(\dl^n_if)(x_0,\dots,\wh{x_{j+1}},\dots,x_{n+2})\\
			&=(\dl^{n+1}_{j+1}\dl^n_if)(x_0,\dots,x_{n+2}).
		\end{align*}
		
		\textit{Case 3.} $0<i\le j$. Then
		\begin{align*}
			(\dl^{n+1}_i\dl^n_jf)(x_0,\dots,x_{n+2})&=(\dl^n_jf)(x_0,\dots,\wh{x_i},\dots,x_{n+2})\\
			&=f(x_0,\dots,\wh{x_i},\dots,\wh{x_{j+1}},\dots,x_{n+2})\\
			&=(\dl^n_if)(x_0,\dots,\wh{x_{j+1}},\dots,x_{n+2})\\
			&=(\dl^{n+1}_{j+1}\dl^n_if)(x_0,\dots,x_{n+2}).
		\end{align*}
	\end{proof}
	
	\begin{lemma}\label{dl^(n+1)_(sg_lb)-circ-dl^n_(sg_lb)=0}
		For all $n\ge 0$ we have $\dl^{n+1}_{(\sg,\lb)}\circ \dl^n_{(\sg,\lb)}=0$.
	\end{lemma}
	\begin{proof}
		Observe that $\dl^{n+1}_{(\sg,\lb)}=\sum_{i=0}^{n+2}(-1)^i\dl^{n+1}_i$ and $\dl^n_{(\sg,\lb)}=\sum_{j=0}^{n+1}(-1)^j\dl^n_j$, so
		\begin{align}
			\dl^{n+1}_{(\sg,\lb)}\circ \dl^n_{(\sg,\lb)}&=\sum_{i=0}^{n+2}\sum_{j=0}^{n+1}(-1)^{i+j}\dl^{n+1}_i\circ \dl^n_j\notag\\
			&=\sum_{0\le i\le j\le n+1}(-1)^{i+j}\dl^{n+1}_i\circ \dl^n_j+\sum_{0\le j<i\le n+2}(-1)^{i+j}\dl^{n+1}_i\circ \dl^n_j.\label{sum_i<=j+sum_j>i}
		\end{align}
		Now, $0\le j<i\le n+2\iff 0\le j\le i-1\le n+1$, so in the second sum of \cref{sum_i<=j+sum_j>i} replace $(j,i-1)$ by $(i',j')$. Then in view of \cref{dl_i.dl_j=dl_(j+1).dl_i} we obtain
		\begin{align*}
			\sum_{0\le j<i\le n+2}(-1)^{i+j}\dl^{n+1}_i\circ \dl^n_j&=\sum_{0\le i'\le j'\le n+1}(-1)^{i'+j'+1}\dl^{n+1}_{j'+1}\circ \dl^n_{i'}\\
			&=-\sum_{0\le i'\le j'\le n+1}(-1)^{i'+j'}\dl^{n+1}_{i'}\circ \dl^n_{j'}.
		\end{align*}
		Thus, we see that \cref{sum_i<=j+sum_j>i} is zero.
	\end{proof}
	
	As a consequence of \cref{dl^n_lb(C^n_lb)-sst-C^(n+1)_lb,dl^(n+1)_(sg_lb)-circ-dl^n_(sg_lb)=0} we obtain the following.
	\begin{theorem}
		The sequence 
		\begin{align}\label{C^1_lb(X_K)->C^2_lb(X_K)->...}
			C^0_{\lb}(X,K)\overset{\dl^0_{(\sg,\lb)}}{\longrightarrow}C^1_{\lb}(X,K)\overset{\dl^1_{(\sg,\lb)}}{\longrightarrow}\cdots
		\end{align}
		is a cochain complex of $K$-spaces.
	\end{theorem}
	
	\begin{definition}
		Let $n\ge 0$. The cochain complex \cref{C^1_lb(X_K)->C^2_lb(X_K)->...} determines the following $K$-spaces of
		\begin{align*}
			\text{\textit{$(\sg,\lb)$-cocycles} }Z^n_{(\sg,\lb)}(X,K)&=\Ker\dl^n_{(\sg,\lb)},\ n\ge 0,\\
			\text{\textit{$(\sg,\lb)$-coboundaries} } B^n_{(\sg,\lb)}(X,K)&=\im\dl^{n-1}_{(\sg,\lb)},\ n>0,\\
			\text{\textit{$(\sg,\lb)$-cohomologies} }H^n_{(\sg,\lb)}(X,K)&=Z^n_{(\sg,\lb)}(X,K)/B^n_{(\sg,\lb)}(X,K),\ n>0,\\
			H^0_{(\sg,\lb)}(X,K)&=Z^0_{(\sg,\lb)}(X,K)
		\end{align*}
		\textit{of degree $n$  of $X$ with values in $K$}, respectively.
	\end{definition}

	\begin{remark}
		Observe that $H^n_{(\zeta,\id)}(X,K)$ is exactly the $n$-th cohomology group of the \textit{order complex} of $X$ (whose $n$-dimensional faces are the chains of length $n$ in $X$) with values in the additive group of $K$ (see, for example,~\cite[\S 5 and \S 42]{Munkres}).
	\end{remark}

	\subsection{Application to skew derivations of $I(X,K)$}
	
	\begin{theorem}\label{psiD/ipsiD-cong-H^1}
		Let $\psi=M_\sg\circ \hat\lb$. Then the space $\psiD(I(X,K))/\ipsiD(I(X,K))$ is isomorphic to $H^1_{(\sg,\lb)}(X,K)$.
	\end{theorem}
	\begin{proof}
		By \cref{funcao_potencial,additive_element}, $(\sg,\lb)$-additive (resp.~$(\sg,\lb)$-potential) elements of $I(X,K)$ are in a one-to-one correspondence with $(\sg,\lb)$-cocycles (resp.~$(\sg,\lb)$-coboundaries) of degree $1$ of $X$ with values in $K$. \cref{ppsiD=ipsiD-cap-apsiD,L_tau-in-vfD} show that the map
		\begin{align*}
			H^1_{(\sg,\lb)}(X,K) &\to \psiD(I(X,K))/\ipsiD(I(X,K)),\\
			f+B^1_{(\sg,\lb)}(X,K)&\mapsto L_f+\ipsiD(I(X,K)),
		\end{align*}
		is well-defined and injective, and \cref{psiD=ipsiD+apsiD} shows that it is also surjective.
	\end{proof}
	
	\begin{corollary}\label{vfD/ivfD-cong-H^1}
		Let $\vf=\xi_\bt\circ M_\sg\circ \hat\lb$. Then the space $\vfD(I(X,K))/\ivfD(I(X,K))$ is isomorphic to $H^1_{(\sg,\lb)}(X,K)$.
	\end{corollary}
	\begin{proof}
		By \cref{vfD(A)=u.psiD(A),psiD/ipsiD-cong-H^1} 
		\begin{align*}
			\vfD(I(X,K))/\ivfD(I(X,K))&=(\bt\cdot\psiD(I(X,K)))/(\bt\cdot\ipsiD(I(X,K)))\\
			&\cong\psiD(I(X,K))/\ipsiD(I(X,K))\\
			&\cong H^1_{(\sg,\lb)}(X,K).
		\end{align*}
	\end{proof}
	
	\begin{corollary}\label{vfD=ivfD-iff-H^1=0}
		Let $\vf=\xi_\bt\circ M_\sg\circ \hat\lb$. Then $\vfD(I(X,K))=\ivfD(I(X,K))$ if and only if $H^1_{(\sg,\lb)}(X,K)$ is trivial.
	\end{corollary}
	
	\subsection{Some properties of the poset skew cohomology}
	
	\begin{definition}
		We say that two multiplicative elements $\sg,\sg'\in I(X,K)$ are \textit{equivalent}, if there exists a fractional element $\eta$ such that $\sg'=\sg*\eta$.
	\end{definition}

	\begin{remark}\label{sg-fractional-iff-sg-equiv-zeta}
		A multiplicative $\sg\in I(X,K)$ is fractional if and only if $\sg$ is equivalent to the zeta function $\zeta\in I(X,K)$.
	\end{remark}
	
	\begin{proposition}\label{H^n_(sg_lb)-cong-H^n_(sg'_lb)}
		If $\sg$ and $\sg'$ are equivalent, then $H^n_{(\sg,\lb)}(X,K)\cong H^n_{(\sg',\lb)}(X,K)$.
	\end{proposition}
	\begin{proof}
		Let $\eta$ be a fractional element such that $\sg'=\sg*\eta$. There exists $\gm:X\to K^*$ such that $\eta(x,y)=\gm(x)\gm(y)\m$ for all $x\le y$ in $X$. For all $n\ge 0$, the map sending $f\in C^n_\lb(X,K)$ to 
		\begin{align*}
			f'(x_0,\dots,x_n)=\gm(\lb(x_0))f(x_0,\dots,x_n)
		\end{align*}
		is clearly an isomorphism of $K$-spaces $C^n_\lb(X,K)\to C^n_\lb(X,K)$. Moreover,
		\begin{align*}
			(\dl^n_{(\sg',\lb)}f')(x_0,\dots,x_{n+1})&=\sg'(\lb(x_0),\lb(x_1))f'(x_1,\dots,x_{n+1})\notag\\
			&\quad+\sum_{i=1}^{n+1}(-1)^i f'(x_0,\dots,\wh{x_i},\dots,x_{n+1})\\
			&=\sg(\lb(x_0),\lb(x_1))\gm(\lb(x_0))\gm(\lb(x_1))\m \gm(\lb(x_1)) f(x_1,\dots,x_{n+1})\notag\\
			&\quad+\sum_{i=1}^{n+1}(-1)^i \gm(\lb(x_0))f(x_0,\dots,\wh{x_i},\dots,x_{n+1})\\
			&=\gm(\lb(x_0))(\dl^n_{(\sg,\lb)}f)(x_0,\dots,x_{n+1})=(\dl^n_{(\sg,\lb)}f)'(x_0,\dots,x_{n+1}).
		\end{align*}
		Thus, $f\mapsto f'$ is an isomorphism of cochain complexes, so it induces an isomorphism of the corresponding cohomology spaces.
	\end{proof}
	
	If $X$ has an all-comparable element, then by \cite[Proposition~7.3.4]{SpDo} any multiplicative $\sg\in I(X,K)$ is fractional. 
	So, in view of \cref{sg-fractional-iff-sg-equiv-zeta} we have the following corollary of \cref{H^n_(sg_lb)-cong-H^n_(sg'_lb)}. 
	
	\begin{corollary}\label{H^1_(sg_lb)-cong-H^1_(zeta_lb)}
		Assume that $X$ has an all-comparable element. Then for any $\lb\in\Aut(X)$ and multiplicative $\sg\in I(X,K)$ we have $H^1_{(\sg,\lb)}(X,K)\cong H^1_{(\zeta,\lb)}(X,K)$.
	\end{corollary}
	
	The following example shows that the converse of \cref{H^n_(sg_lb)-cong-H^n_(sg'_lb)} does not hold.
	\begin{example}\label{2_crown_1}
		Let $X=\{1,2,3,4\}$ with the following Hasse diagram ($2$-crown poset).
		
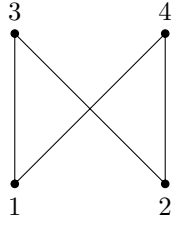
\begin{center}
			\begin{tikzpicture}
				\draw [fill=black] (0,-0.3) node{$1$}   (0,2.3) node{$3$}  (2,-0.3) node{$2$}  (2,2.3) node{$4$}
				(1.02,-0.5) (0,0) circle(0.05)--(0,1)--(0,2)circle(0.05)--(1,1)--(2,0)circle(0.05)--(2,1)--(2,2)circle(0.05) --(1,1)--(0,0);
			\end{tikzpicture}
		\end{center} 
		Take an arbitrary multiplicative $\sg\in I(X,K)$ and $\lb\in\Aut(X)$ such that $\lb(1)=2$, $\lb(2)=1$, $\lb(3)=4$ and $\lb(4)=3$. Then $\lb(x)\nleq x$ for all $x\in X$, so $C^0_\lb(X,K)=\{0\}$, and consequently $B^1_{(\sg,\lb)}(X,K)=\{0\}$. Hence, $H^1_{(\sg,\lb)}(X,K)\cong Z^1_{(\sg,\lb)}(X,K)$, the latter being the same as $Z^1_{(\sg,\id)}(X,K)\cong K^4$, because $x<y\iff\lb(x)<y$ for all $x,y\in X$. Thus, $H^1_{(\sg,\lb)}(X,K)$ does not depend on $\sg$, in particular, $H^1_{(\sg,\lb)}(X,K)\cong H^1_{(\zeta,\lb)}(X,K)$, but there exist $\sg$ that are not equivalent to $\zeta$ (see, for instance, \cite[Theorem 7.3.6 and Proposition 7.3.3]{SpDo} and \cite[Example]{Stanley70}).
		%$H^1_{(\sg,\lb)}(X,K)\cong K^4$ for any multiplicative $\sg\in I(X,K)$.
	\end{example}
	
	However, the condition that $\sg$ and $\sg'$ are equivalent in \cref{H^n_(sg_lb)-cong-H^n_(sg'_lb)} cannot be dropped, as the next example shows.
	\begin{example}\label{H^1_(sg_id)(2-crown)}
		Let $X=\{1,2,3,4\}$ as in \cref{2_crown_1}. Then
		\begin{align*}
			H^1_{(\sg,\id)}(X,K)\cong
			\begin{cases}
				K, & \mbox{if }\sg\mbox{ is fractional},\\
				\{0\}, & \mbox{otherwise}.
			\end{cases}
		\end{align*}
		
		Indeed, $C^1_\id(X,K)=C^1(X,K)$ and
		\begin{align}\label{Z^1_(sg_lb)(X_K)={f|f(x_x)=0}}
			Z^1_{(\sg,\id)}(X,K)=\{f\in C^1(X,K) : f(x,x)=0 \text{ for all } x\in X\}\cong K^4.    
		\end{align}
		Given $f\in Z^1_{(\sg,\id)}(X,K)$, there exists $g\in C^0_\id(X,K)$ with $f=\dl^0_{(\sg,\id)}g$ if and only if the linear system
		\begin{align}
			f(1,3)&=\sg(1,3)g(3)-g(1),\ f(1,4)=\sg(1,4)g(4)-g(1),\notag\\
			f(2,3)&=\sg(2,3)g(3)-g(2),\ f(2,4)=\sg(2,4)g(4)-g(2)\label{system-f=dl^0g}
		\end{align}
		admits a solution in $g(1), g(2), g(3), g(4)$. The determinant of \cref{system-f=dl^0g} equals 
		\begin{align*}
			\Dl=\sg(1,3)\sg(2,4)-\sg(1,4)\sg(2,3).
		\end{align*}
		Observe that
		\begin{align*}
			\Dl=0\iff\sg(1,3)\sg(2,3)\m\sg(2,4)\sg(1,4)\m=1,
		\end{align*}
		the latter being equivalent to fractionality of $\sg$ (see \cite[Theorem~5]{BFS15}). Thus, if $\sg$ is not fractional, then $Z^1_{(\sg,\id)}(X,K)=B^1_{(\sg,\id)}(X,K)$, so $H^1_{(\sg,\id)}(X,K)=\{0\}$. Otherwise, 
		\begin{align*}
			&\sg(2,4)(\sg(1,3)g(3)-g(1))-\sg(2,4)(\sg(1,4)g(4)-g(1))\\
			&\quad-\sg(1,4)(\sg(2,3)g(3)-g(2))+\sg(1,4)(\sg(2,4)g(4)-g(2))=0,
		\end{align*}
		so if \cref{system-f=dl^0g} has a solution, then
		\begin{align*}
			\sg(2,4)f(1,3)-\sg(2,4)f(1,4)-\sg(1,4)f(2,3)+\sg(1,4)f(2,4)=0.
		\end{align*}
		The converse is also true, because any subsystem of \cref{system-f=dl^0g} consisting of three equations has a solution.
		Hence, $B^1_{(\sg,\id)}(X,K)\cong K^3$, and $H^1_{(\sg,\id)}(X,K)\cong K$ in this case.
		%\end{proof}
	\end{example}
	
	\begin{proposition}\label{all-comp=>H^1_(sg_lb)=0}
		Assume that $X$ has an all-comparable element $x_0$. If $\lb\in\Aut(X)$ is such that 
		\begin{align}\label{x_0_lb(x_0)<=x_or_x_lb(x)<=x_0-impl-lb(x)<=x}
			(x_0,\lb(x_0)\le x\text{ or }x,\lb(x)\le x_0)\ \impl\ \lb(x)\le x,
		\end{align}
		then $H^1_{(\sg,\lb)}(X,K)=\{0\}$ for any multiplicative $\sg\in I(X,K)$.
	\end{proposition}
	\begin{proof}
		In view of \cref{H^1_(sg_lb)-cong-H^1_(zeta_lb)} it suffices to consider the case $\sg=\zeta$. Let $f\in Z^1_{(\zeta,\lb)}(X,K)$ and define $g:X\to K$ by
		\begin{align}\label{g-in-terms-of-f}
			g(x)=
			\begin{cases}
				f(x_0,x),& \mbox{if } x_0\le x,\\
				-f(x,x_0), & \mbox{if } x<x_0.
			\end{cases}
		\end{align}
		Assume that $\lb(x)\nleq x$. If $x_0\le x$, then $\lb(x_0)\nleq x$, so $f(x_0,x)=0$. Otherwise $x<x_0$, whence $\lb(x)\nleq x_0$, so $f(x,x_0)=0$. In both cases, $g(x)=0$. Thus, $g\in C^0_\lb(X,K)$.
		
		We now prove that $\dl^0_{(\zeta,\lb)}g=f$, i.e., for arbitrary $x\le y$,
		\begin{align*}
			g(y)-g(x)=f(x,y).
		\end{align*}
		Since $x_0$ is comparable to $x$ and $y$, there are $3$ cases to consider.
		
		\textit{Case 1.} $x_0\le x\le y$. Then 
		\begin{align}\label{f(x_0_y)=f(x_0_x)+sg(lb(x_0)_lb(x))f(x_y)}
			f(x_0,y)=f(x_0,x)+f(x,y).
		\end{align}
		By \cref{g-in-terms-of-f,f(x_0_y)=f(x_0_x)+sg(lb(x_0)_lb(x))f(x_y)} we have
		\begin{align*}
			g(y)-g(x)=f(x_0,y)-f(x_0,x)=f(x,y).
		\end{align*}
		
		\textit{Case 2.} $x< x_0\le y$. Then 
		\begin{align}\label{f(x_y)=f(x_x_0)+sg(lb(x)_lb(x_0))f(x_0_y)}
			f(x,y)&=f(x,x_0)+f(x_0,y).
		\end{align}
		By \cref{g-in-terms-of-f,f(x_y)=f(x_x_0)+sg(lb(x)_lb(x_0))f(x_0_y)} we have
		\begin{align*}
			g(y)-g(x)=f(x_0,y)+f(x,x_0)=f(x,y).
		\end{align*}
		
		\textit{Case 3.} $x\le y< x_0$. Then 
		\begin{align}\label{f(x_x_0)=f(x_y)+sg(lb(x)_lb(y))f(y_x_0)}
			f(x,x_0)=f(x,y)+f(y,x_0).
		\end{align}
		By \cref{g-in-terms-of-f,f(x_x_0)=f(x_y)+sg(lb(x)_lb(y))f(y_x_0)} we have
		\begin{align*}
			g(y)-g(x)=-f(y,x_0)+f(x,x_0)=f(x,y).
		\end{align*}
	\end{proof}
	
	As a consequence of \cref{all-comp=>H^1_(sg_lb)=0,vfD=ivfD-iff-H^1=0}, we obtain an analogue of \cite[Theorem 7.1.9]{SpDo}.
	
	\begin{corollary}
		Assume that $X$ has an all-comparable element $x_0$.  Let $\vf=\xi_\bt\circ M_\sg\circ \hat\lb$, where $\lb$ satisfies \cref{x_0_lb(x_0)<=x_or_x_lb(x)<=x_0-impl-lb(x)<=x}. Then $\vfD(I(X,K))=\ivfD(I(X,K))$.
	\end{corollary}

	\begin{remark}\label{lb(x)=x-for-l(X)<infty}
		Let $l(X)<\infty$. Then for all $\lb\in\Aut(X)$ and $x\in X$
		\begin{align*}
			\lb(x)\le x\iff x\le\lb(x)\iff\lb(x)=x.
		\end{align*}
		In particular, if $x_0\in X$ is all-comparable, then $\lb(x_0)=x_0$.
		
		For, if $\lb(x)<x$ (resp.~$\lb(x)>x$), then $\{\lb^n(x)\}_{n=0}^\infty$ is an infinite descending (resp.~ascending) chain in $X$ contradicting $l(X)<\infty$.
	\end{remark}
	
	\begin{remark}\label{all-comp-l(X)<infty}
		Let $l(X)<\infty$ and assume that $X$ has an all-comparable element. Then \cref{x_0_lb(x_0)<=x_or_x_lb(x)<=x_0-impl-lb(x)<=x} is equivalent to $\lb=\id$.
		
		Indeed, since $\lb(x_0)=x_0$ by \cref{lb(x)=x-for-l(X)<infty}, we have $x_0,\lb(x_0)\le x\iff x_0\le x$, and similarly $x,\lb(x)\le x_0\iff x\le x_0$. Hence, the premise of \cref{x_0_lb(x_0)<=x_or_x_lb(x)<=x_0-impl-lb(x)<=x} is always true, so $\lb(x)\le x$ for all $x\in X$. By \cref{lb(x)=x-for-l(X)<infty} we conclude that $\lb=\id$.
	\end{remark}
	
	The following example shows that the condition \cref{x_0_lb(x_0)<=x_or_x_lb(x)<=x_0-impl-lb(x)<=x} in \cref{all-comp=>H^1_(sg_lb)=0} cannot be dropped.
	\begin{example}\label{V-poset}
		Let $X=\{1,2,3\}$ with the following Hasse diagram.
		
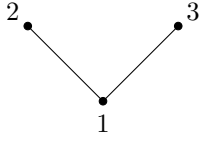
\begin{center}
			\begin{tikzpicture}
				\draw  (0,0)-- (-1,1);
				\draw  (0,0)-- (1,1);
				\draw [fill=black] (-1,1) circle (0.05);
				\draw  (-1.2,1.2) node {$2$};
				\draw [fill=black] (1,1) circle (0.05);
				\draw  (1.2,1.2) node {$3$};
				\draw [fill=black] (0,0) circle (0.05);
				\draw  (0,-0.3) node {$1$};
			\end{tikzpicture}
		\end{center}
		Take $\lb\in\Aut(X)$ such that $\lb(1)=1$, $\lb(2)=3$ and $\lb(3)=2$. Then $H^1_{(\sg,\lb)}(X,K)\cong K$ for all multiplicative $\sg\in I(X,K)$. 
		
		Indeed, since $l(X)=1$, then  
		\begin{align}\label{Z^1_(sg_lb)(X_K)={f|f(1_1)=f(2_2)=f(3_3)=0}}
			Z^1_{(\sg,\lb)}(X,K)=\{f\in C^1_\lb(X,K) : f(1,1)=f(2,2)=f(3,3)=0\}.    
		\end{align}
		Moreover, $\lb(1)<2$, $\lb(1)<3$, $\lb(1)=1$, $\lb(2)\nleq 2$ and $\lb(3)\nleq 3$, so 
		\begin{align*}
			C^1_\lb(X,K)=\{f\in C^1(X,K) : f(2,2)=f(3,3)=0\}.    
		\end{align*}
		It follows that in the right-hand side of \cref{Z^1_(sg_lb)(X_K)={f|f(1_1)=f(2_2)=f(3_3)=0}} we can replace $C^1_\lb(X,K)$ by $C^1(X,K)$, so $Z^1_{(\sg,\lb)}(X,K)\cong K^2$.
		
		Let $f\in Z^1_{(\sg,\lb)}(X,K)$. Assume that $f=\dl^0_{(\sg,\lb)}g$ for some $g\in C^0_\lb(X,K)$. We have $\lb(x)\le x$ if and only if $x=1$, so $g(2)=g(3)=0$. Hence,
		\begin{align*}
			f(1,2)=\sg(\lb(1),\lb(2))g(2)-g(1)=-g(1)=\sg(\lb(1),\lb(3))g(3)-g(1)=f(1,3).
		\end{align*}
		Conversely, if $f(1,2)=f(1,3)$, then $f=\dl^0_{(\sg,\lb)}g$, where $g(1)=-f(1,2)$ and $g(2)=g(3)=0$. Thus, $B^1_{(\sg,\lb)}(X,K)\cong K$ and $H^1_{(\sg,\lb)}(X,K)\cong K^2/K\cong K$. 
	\end{example}
	
	\begin{remark}\label{H^1_(zeta_lb)=K-and-H^1_(zeta_id)=0}
		In the conditions of \cref{V-poset} we have $H^1_{(\zeta,\lb)}(X,K)\cong K$.\footnote{Notice that it also follows from \cref{H^1_(sg_lb)-cong-H^1_(zeta_lb)} that $H^1_{(\sg,\lb)}(X,K)\cong H^1_{(\zeta,\lb)}(X,K)$ for any multiplicative $\sg$.} However, $H^1_{(\zeta,\id)}(X,K)=\{0\}$ in view of \cref{all-comp=>H^1_(sg_lb)=0,all-comp-l(X)<infty} (see also \cite[Proposition~7.1.8]{SpDo}).
	\end{remark}
	
	We proceed with an example showing that $\lb\ne\id$ can result in $H^1_{(\sg,\lb)}(X,K)=\{0\}$ in a case where $H^1_{(\sg,\id)}(X,K)\ne\{0\}$ (a situation in certain sense opposite to that of \cref{H^1_(zeta_lb)=K-and-H^1_(zeta_id)=0}).
	
	\begin{example}\label{4-crown-poset}
		Let $X=\{1,2,3,4,5,6,7,8\}$ with the following Hasse diagram ($4$-crown poset).
		
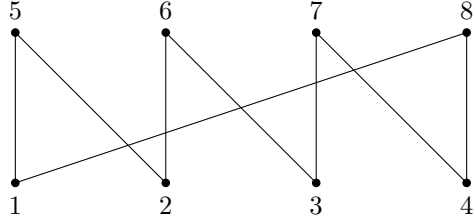
\begin{center}
			\begin{tikzpicture}[line cap=round,line join=round,>=triangle 45,x=1cm,y=1cm]
				\draw  (0,0)-- (0,2);
				\draw  (0,2)-- (2,0);
				\draw  (2,0)-- (2,2);
				\draw  (2,2)-- (4,0);
				\draw  (4,0)-- (4,2);
				\draw  (4,2)-- (6,0);
				\draw  (6,0)-- (6,2);
				\draw  (0,0)-- (6,2);
				\draw [fill=black] (0,0) circle (0.05);
				\draw  (0,-0.3) node {$1$};
				\draw [fill=black] (0,2) circle (0.05);
				\draw  (0,2.3) node {$5$};
				\draw [fill=black] (2,0) circle (0.05);
				\draw  (2,-0.3) node {$2$};
				\draw [fill=black] (2,2) circle (0.05);
				\draw  (2,2.3) node {$6$};
				\draw [fill=black] (4,0) circle (0.05);
				\draw  (4,-0.3) node {$3$};
				\draw [fill=black] (4,2) circle (0.05);
				\draw  (4,2.3) node {$7$};
				\draw [fill=black] (6,0) circle (0.05);
				\draw  (6,-0.3) node {$4$};
				\draw [fill=black] (6,2) circle (0.05);
				\draw  (6,2.3) node {$8$};
			\end{tikzpicture}
		\end{center}
		Consider $\lb\in\Aut(X)$ such that $\lb(1)=3$, $\lb(2)=4$, $\lb(3)=1$, $\lb(4)=2$, $\lb(5)=7$, $\lb(6)=8$, $\lb(7)=5$ and $\lb(8)=6$. Then $Z^1_{(\sg,\lb)}(X,K)=\{0\}$ for any multiplicative $\sg\in I(X,K)$, because $\{(x,y)\in X^2_\le : \lb(x)\le y\}=\emptyset$. In particular, $H^1_{(\zeta,\lb)}(X,K)=\{0\}$, but $H^1_{(\zeta,\id)}(X,K)\ne\{0\}$ by \cite[Theorem 3.11]{FP1}.\footnote{In fact, $H^1_{(\zeta,\id)}(X,K)\cong K$, because $\tau\in Z^1_{(\zeta,\id)}(X,K)$ belongs to $B^1_{(\zeta,\id)}(X,K)$ if and only if $\tau(1,5)-\tau(2,5)+\tau(2,6)-\tau(3,6)+\tau(3,7)-\tau(4,7)+\tau(4,8)-\tau(1,8)=0$.} 
	\end{example}
	
	We can slightly modify \cref{4-crown-poset} to show that condition \cref{x_0_lb(x_0)<=x_or_x_lb(x)<=x_0-impl-lb(x)<=x} from \cref{all-comp=>H^1_(sg_lb)=0} is not necessary for $H^1_{(\sg,\lb)}(X,K)=\{0\}$.
	\begin{example}\label{4-crown-with-0}
		Let $X$ be the poset from \cref{4-crown-poset} with adjoint minimum element $0$. The map $\lb$ from \cref{4-crown-poset} uniquely extends to an automorphism of $X$ by $\lb(0)=0$. Let $f\in Z^1_{(\sg,\lb)}(X,K)$. Then $f(x,y)=0$ for all $x,y\in\{1,\dots, 8\}$, $x\le y$, as in \cref{4-crown-poset}. Hence,
		\begin{align*}
			f(0,1)=f(0,1)+\sg(0,3)f(1,5)=f(0,5)=f(0,2)+\sg(0,4)f(2,5)=f(0,2).
		\end{align*}
		Similarly, $f(0,2)=f(0,6)=f(0,3)$, $f(0,3)=f(0,7)=f(0,4)$ and $f(0,4)=f(0,8)=f(0,1)$. Define $g(0)=-f(0,1)$ and $g(x)=0$ for all $x\ne 0$. Then clearly $g\in C^0_\lb(X,K)$ and $f=\dl^0_{(\sg,\lb)}g$. Thus, $H^1_{(\sg,\lb)}(X,K)=\{0\}$.
	\end{example}

	%\section*{Statements and Declarations}
	
%	\textbf{Competing Interests:} There is no competing interest.

	\section*{Acknowledgements}
	
	The second author was partially supported by CMUP, member of LASI, which is financed by national funds through FCT --- Fundação para a Ciência e a Tecnologia, I.P., under the project with reference UIDB/00144/2020. The authors are grateful to the referee for a very careful reading of the paper and the suggested useful improvements.
	\bibliography{bibl}{}
	\bibliographystyle{acm}
	
\end{document}